\documentclass[review,sort&compress, times,12pt]{elsarticle}

\usepackage{amsmath, amssymb, amsthm, bm, mathrsfs, color, url}

\newtheorem{theorem}{Theorem}[section]

\newtheorem{definition}[theorem]{Definition}

\newtheorem{lemma}[theorem]{Lemma}

\newcommand{\dom}{\mathop{\rm dom}\nolimits}

\begin{document}

\begin{frontmatter}

\title{Stability of Infinite-dimensional Sampled-data  Systems
	with Unbounded Control Operators and Perturbations\tnoteref{title}}
\tnotetext[title]{
This work was supported in part by JSPS KAKENHI Grant Number
JP20K14362.}
\author[add1]{Masashi~Wakaiki\corref{cor1}}
\ead{wakaiki@ruby.kobe-u.ac.jp}
\cortext[cor1]{Corresponding author. Tel.:~+81 78 803 6232; fax: +81 78 803 6392}
\address[add1]{Department of Applied Mathematics, Graduate School of System Informatics, Kobe University, Kobe
657-8501, Japan}

\begin{abstract}
	We  analyze the robustness of the exponential stability of
	infinite-dimensional sampled-data systems with unbounded control
	operators.
	The unbounded perturbations we consider are the so-called
	Desch-Schappacher perturbations, which arise, e.g., from
	the boundary perturbations of systems described by
	partial differential equations.
	As the main result,
	we show that the exponential stability of the sampled-data system
	is preserved under all Desch-Schappacher perturbations
	sufficiently small in a certain sense.
\end{abstract}

\end{frontmatter}

\section{Introduction}
Consider the following sampled-data system with state space $X$
and input space $U$ (both Banach spaces):
\begin{equation}
\label{eq:plant}
\left\{
\begin{aligned}
\dot x(t) &= (A+D)x(t)+Bu(t), \quad t\geq 0;\qquad 
x(0)= x^0 \in X \\
u(t) &= Fx(k\tau),\quad  k\tau\leq t < (k+1)\tau,~ k=0,1,2,\dots,
\end{aligned}
\right.
\end{equation}
where $\tau>0$ is the sampling period, 
$A$ is the generator of 
a strongly continuous semigroup $(T(t))_{t\geq 0}$ on $X$, and
the feedback operator $F$ is a bounded linear operator from $X$
to $U$.
The control operator $B$
is linear but unbounded in the sense that 
it maps $U$  into  a larger space than $X$.
More precisely, we employ the extrapolation space $X_{-1}$,
which is the completion of $X$ with respect to an appropriate norm.
The perturbation $D$ is also linear but unbounded, which maps 
$X$ into $X_{-1}$. In particular, we
are interested in the so-called Desch-Schappacher perturbations.
See Section~\ref{sec:preliminaries} for the details of these notions.
If $B$ and $D$ map boundedly into $X$, then they are called bounded; 
otherwise they are called unbounded. 
In this paper,  we study the exponential stability of the perturbed
sampled-data system \eqref{eq:plant}.

A fundamental issue to be considered before discussing the 
stability of the perturbed system is 
the robustness of the generation of strongly
continuous semigroups:
Given the generator $A$ of a strongly continuous semigroup, for which
perturbation $D$ does the (properly-defined) sum $A+D$ again
become a generator?
It is well known that an affirmative answer has been given
to this question for some classes of perturbations such as
relatively $A$-bounded perturbations \cite[Chapter~XIII]{Hille1957}, 
\cite[Section~III.2]{Engel2000}, 
Miyadera-Voigt perturbations \cite{Miyadera1966,Voigt1977}, \cite[Section~III.3.c]{Engel2000}, and 
Desch-Schappacher perturbations \cite{Desch1989}, \cite[Section~III.3.a]{Engel2000}.
In this direction, Weiss \cite{Weiss1994} 
and Staffans \cite[Chapter~7]{Staffans2005} have presented
generation results on $A+D$ for
a class of perturbations $D$ factorized as $D=D_1D_2$ 
with an admissible 
control operator $D_1$ and 
an admissible observation operator $D_2$.
The idea of factorizing perturbations has been also
used in \cite{Haak2006} in order to develop 
robustness results on sectoriality and
maximal $L^p$-regularity under more general perturbations with small norm.
The robustness analysis of exponential stability 
for infinite-dimensional 
continuous-time systems has been performed
under several classes of unbounded perturbations
in \cite{Pritchard1989, Pandolfi1991, Townley1992Survey, PaunonenMasterThesis}.
In the discrete-time setting,
time-varying perturbations have been considered for
the exponential stability of infinite-dimensional  time-varying 
systems in \cite{Wirth1994, Sasu2004}.
Strong stability and polynomial stability are much weaker notions of
stability than exponential stability, and 
perturbations preserving these nonexponential classes 
of  stability
have been investigated, e.g., in \cite{Paunonen2011,
	Paunonen2012SS,Paunonen2014JDE,Paunonen2015}.

For sampled-data systems, the robustness of stability 
with respect to sampling and perturbations have
been studied.
Robustness with respect to sampling means that 
the closed-loop stability is preserved when an idealized sample-and-hold
process is applied to a stabilizing continuous-time controller.
For infinite-dimensional systems,
this type of robustness has been discussed
 in \cite{Rebarber2002, Logemann2003, Rebarber2006,Wakaiki2021SIAM}.
Robustness analysis with respect to perturbations has been also developed 
for
infinite-dimensional sampled-data systems in \cite{Wakaiki2020SCL,Wakaiki2020_self}.
Bounded control operators and unbounded perturbations 
have been considered in
\cite{Wakaiki2020SCL}, while unbounded control operators and bounded nonlinear perturbations
have been treated in \cite[Lemma~4.5]{Wakaiki2020_self}.

In this paper, we continue and expand the robustness analysis developed
in \cite{Wakaiki2020SCL}.
Desch-Schappacher perturbations, in which we are interested here, 
have properties different
 from relatively $A$-bounded perturbations and 
 Miyadera-Voigt perturbations considered in \cite{Wakaiki2020SCL}. In fact,
Desch-Schappacher perturbations map  the state space $X$ into 
the extrapolation space $X_{-1}$, whereas the
perturbations studied in \cite{Wakaiki2020SCL} map 
a subspace of $X$, e.g., the domain of $A$, into $X$. 
The class of Desch-Schappacher perturbations is not a subset of
the classes of perturbations considered in \cite{Wakaiki2020SCL}
and vice versa.
Operators associated with boundary conditions of partial differential equations
are often represented by using  $X_{-1}$.
Therefore, Desch-Schappacher perturbations appear when boundary 
conditions of partial differential equations are subject to perturbations, as
shown in Section~\ref{sec:example}.
Moreover, this class of perturbations is also used for
dynamical population equations and delay differential 
equations; see, for example, 
\cite{Piazzera2004,Maniar2005} and 
\cite[Section~VI.6]{Engel2000}.
Robustness of controllability under Desch-Schappacher perturbations
has been studied in \cite{Boulite2005, Hadd2005JEEQ}.

Our aim is to prove that the
exponential stability of sampled-data systems
with unbounded control operators
is preserved under sufficiently small Desch-Schappacher perturbations.
The technical difficulties of this study come from
the combination of the unbounded 
control operator and perturbation.
In particular, we have to be careful about 
the operator $S_{\!D}(\tau)$ on $U$ defined by
\begin{equation*}
S_{\!D}(\tau)u :=
\int^\tau_0 T_D(s)B u\mathrm{d}s,\quad u \in U,
\end{equation*}
where $(T_D(t))_{t \geq 0}$ is the semigroup generated by the sum 
$A+D$, since
$S_{\!D}(\tau)$ contains both
the control operator $B$ and the perturbation $D$. 
The arguments used in \cite{Wakaiki2020SCL,Wakaiki2020_self}
rely on
the assumption that either $B$ or $D$ is bounded.
To treat the case where 
$B$ and $D$ are both unbounded,
we develop an alternative approach which
employs 
the fact that
the resolvent of $A$ is extended to a bounded operator 
from $X_{-1}$ to $X$ and hence 
absorbs the unboundedness of $B$ and $D$.

This paper is organized as follows.
In Section~\ref{sec:preliminaries}, we present
preliminaries on extrapolation spaces, Desch-Schappacher perturbations, and
the solution of the abstract evolution equation \eqref{eq:plant}.
Section~\ref{sec:main_result} contains our main result and its proof.
Section~\ref{sec:example} is devoted to an application to
the boundary control of a heated rod with a boundary perturbation,
which is extended to a diagonal system in a Banach framework.

\paragraph*{Notation}
Let $\mathbb{Z}_+$ and $\mathbb{R}_+$ denote the set of nonnegative integers
and the set of nonnegative real numbers, respectively.
Let $X$ and $Y$ be Banach spaces. 
The space of
all bounded linear operators from $X$ to $Y$ is denoted by
$\mathcal{L}(X,Y)$. We set $\mathcal{L}(X) := \mathcal{L}(X,X)$.
For a linear operator $A$ defined in $X$ and mapping to $Y$, 
the domain of $A$ is denote by $\dom (A)$.
We denote by $\varrho(A)$
the resolvent set of a linear operator $A:\dom(A) \subset X \to X$.
Let $J$ be an interval of $\mathbb{R}$.
We denote by $C(J,X)$ and $C^1(J,X)$  
the space of all continuous functions $f :J \to X$
and the space of all continuously differential functions $f:J \to X$,
respectively. For $1 \leq p < \infty$, we denote by $L^p(J,X)$ 
the space of all measurable functions $f : J \to X$ such that 
$\int_J \|f(t)\|_X^p \mathrm{d}t < \infty$. We write $L^p(J)$ for
$L^p(J,\mathbb{C})$.

\section{Preliminaries}
\label{sec:preliminaries}
\subsection{Extrapolation spaces}
Let $A$ be the generator of a strongly continuous
semigroup $(T(t))_{t\geq 0}$ on a Banach space $X$.
We denote by $\|\cdot\|$ the norm on $X$
and introduce a new norm
\[
\|x\|_{-1}^A := \|(\lambda I -A)^{-1} x\|,\quad  x \in X
\]
for some $\lambda \in \varrho(A)$.
The completion of $X$ with respect to this norm is called 
the {\em extrapolation space}
associated with $A$ (or $(T(t))_{t \geq 0}$) and is denoted by $X_{-1}^A$.
Different choices of $\lambda$ lead to equivalent norms,
and hence $X_{-1}^A$ is unique up to isomorphism.
For simplicity, we shall denote  the
space $X_{-1}^A$ associated with $A$ 
by $X_{-1}$. 

For $t \geq 0$,
let $T_{-1}(t)$ be the continuous extension of the operator $T(t)$
to the extrapolation $X_{-1}$. Then the operators
$(T_{-1}(t))_{t\geq 0}$ form a strongly continuous semigroup on $X_{-1}$.
The domain of the 
generator $A_{-1}$ of $(T_{-1}(t))_{t\geq 0}$ is given by 
$\dom(A_{-1}) = X$, and
$A_{-1}$ is the unique continuous extension of $A: \dom (A) \to X$ to 
$\mathcal{L}(X,X_{-1})$.
Moreover, $\varrho(A) = \varrho(A_{-1})$ holds.
See Section~II.5 in \cite{Engel2000} and
Section~2.10 in \cite{Tucsnak2009} for more details
on the extrapolation space.

\subsection{Desch-Shappacher perturbations}
Let $X$ be a Banach space and take $t_0 >0$. We denote by
$\mathcal{X}_{t_0}$ the space of
all functions on $[0,t_0]$ into $\mathcal{L}(X)$ 
that are continuous for the strong operator topology.
Then $\mathcal{X}_{t_0}$ is a Banach space with the norm
\[
\|Q\|_{\infty} := \sup_{t \in [0,t_0]} \|Q(t)\|_{\mathcal{L}(X)},\quad 
Q \in \mathcal{X}_{t_0}.
\]
Let $(T(t))_{t\geq 0}$ be a 
strongly continuous
semigroup  on $X$.
For $\Psi \in \mathcal{L}(X,X_{-1})$, we define
the abstract Volterra operator $V_{\Psi}$ on $\mathcal{X}_{t_0}$:
\[
(V_{\Psi}Q)(t) x:=
\int^t_0 T_{-1}(t-s) \Psi Q(s)x\mathrm{d}s,\quad x \in X,~t\in[0,t_0],~Q \in \mathcal{X}_{t_0}.
\]	
This integral converges in $X_{-1}$.
In general, we only have
$(V_{\Psi}Q)(t) \in \mathcal{L}(X,X_{-1})$ for $0\leq t \leq t_0$.  Hence
the Volterra operator $V_{\Psi}$ may not be bounded on $\mathcal{X}_{t_0}$.
The 
class of
{\em Desch-Shappacher perturbations} $\mathcal{S}^{\rm DS}_{t_0}$ is defined by
\[
\mathcal{S}^{\rm DS}_{t_0} := 
\{
\Psi \in \mathcal{L}(X,X_{-1}) : V_{\Psi} \in \mathcal{L} (\mathcal{X}_{t_0}),~
\|V_{\Psi}\|_{\mathcal{L}(\mathcal{X}_{t_0})}  < 1
\}.
\]

We review some  important properties of 
Desch-Shappacher perturbations. The following theorem states that 
the generation of 
strongly continuous semigroups is preserved under 
Desch-Shappacher perturbations.  For the proof,
see Theorem~III.3.1 and Corollary~III.3.2 in \cite{Engel2000}
\begin{theorem}
	\label{thm:generation}
	Let $A$ be the generator of a strongly continuous semigroup $(T(t))_{t\geq0}$
	on a Banach space $X$. If $D \in \mathcal{S}^{\rm DS}_{t_0} $
	for some $t_0 >0$, then the operator $A_D$ defined by
	\begin{equation}
	\label{eq:AD_def}
	A_D x:=(A_{-1}+D)x,\quad x \in
	\dom(A_D) := \{x \in X: (A_{-1}+D)x \in X\}
	\end{equation}
	generates
	a strongly continuous semigroup $(T_D(t))_{t\geq 0}$ on $X$. 
	The semigroup $(T_D(t))_{t\geq 0}$ satisfies the 
	following variation of constants formula:
	\begin{equation}
		\label{eq:VCF}
		T_D(t) x = 
		T(t)x + \int^t_0 T_{-1}(t-s) DT_D(s)x \mathrm{d}s
	\end{equation}
	for all $x \in X$ and all $t\geq 0$.
\end{theorem}

One can easily see that if $D_0 \in \mathcal{S}^{\rm DS}_{t_0}$, then
$D = cD_0$ with $0\leq c \leq 1 $ also satisfies 
$D\in \mathcal{S}^{\rm DS}_{t_0}$. Moreover,
the next theorem is helpful 
to verify the property $D \in \mathcal{S}^{\rm DS}_{t_0} $ in concrete examples; see Corollary III.3.4 of \cite{Engel2000}
for the proof.
\begin{theorem}
	\label{thm:verify}
Let $(T(t))_{t\geq0}$ be a strongly continuous semigroup 
on a Banach space $X$ and let $D \in \mathcal{L}(X,X_{-1})$.
If there exist $t_1 >0$ and $p \in [1,\infty)$ such that 
\[
\int^{t_1}_0 T_{-1}(t-s) Df(s)\mathrm{d}s \in X 
\]
for all functions $f \in L^p([0,t_1],X)$, then
$D \in \mathcal{S}^{\rm DS}_{t_0} $ for some $t_0 >0$.
\end{theorem}

The following result holds for
the resolvent of
the sum $A_D$; see
(iv) of the proof of Theorem~III.3.1 in \cite{Engel2000}.
\begin{lemma}
	\label{lem:RD_bound}
Let $A$ be the generator of a strongly continuous semigroup $(T(t))_{t\geq0}$
on a Banach space $X$. If $D \in \mathcal{S}^{\rm DS}_{t_0} $
for some $t_0 >0$, then
there exists $\lambda^*> 0$ such that for all
$\lambda \geq \lambda^*$, one has $\lambda \in \varrho(A) \cap \varrho(A_D)$,
\begin{equation*}
\|(\lambda I - A_{-1})^{-1}D\|_{\mathcal{L}(X)} < 1,
\end{equation*}
and
\begin{equation*}
(\lambda I - A_{D})^{-1} = (I - (\lambda I - A_{-1})^{-1}D)^{-1} (\lambda I - A)^{-1}.
\end{equation*}
\end{lemma}

Let $t_0 >0$ and take $D \in \mathcal{S}^{\rm DS}_{t_0}$.
By Lemma~\ref{lem:RD_bound}, the norms on $X$
\[
\|x\|_{-1}^A =
\|(\lambda I - A)^{-1}x\|,\quad \|x\|_{-1}^{A_D} = \|(\lambda I - A_D)^{-1}x\|
\]
are equivalent for sufficiently large $\lambda>0$.
This implies that 
the extrapolation space $X_{-1}^{A_D}$ associated with $A_D$ 
is isomorphic to $X_{-1}$. In what follows, we identify 
$X_{-1}^{A_D}$ with $X_{-1}$. For $t \geq 0$,
let $(T_D)_{-1}(t)$ be the continuous extension of the operator $T_D(t)$
to the extrapolation space $X_{-1}$.
The generator $(A_D)_{-1}$ 
of the semigroup $((T_D)_{-1}(t))_{t \geq 0}$
is the unique continuous extension of $A_D:\dom(A_D) \to X$
to $\mathcal{L}(X,X_{-1})$.
Since $A_{-1}+D$ also belongs to $\mathcal{L}(X,X_{-1})$, it follows 
from the definition \eqref{eq:AD_def} 
of $A_D$ that 
\begin{equation}
\label{eq:AD_identity}
(A_D)_{-1}x = (A_{-1}+D)x\qquad \forall x \in X.
\end{equation}

\subsection{Solution of abstract evolution equation \eqref{eq:plant}}
Let $B \in \mathcal{L}(U,X_{-1})$, and
define the operator $S(t)$ on $U$ by
\begin{equation}
	\label{eq:S_def}
S(t)u := \int^t_0 T_{-1}(s)Bu \mathrm{d}s,\quad u \in U,~t \geq 0.
\end{equation}
For the analysis of the abstract evolution equation \eqref{eq:plant},
we recall the results
developed in Lemma~2.2 and its proof of \cite{Logemann2003},
where $X$ and $U$ are Hilbert spaces.
The idea of the proof is to use the following 
standard fact on strongly continuous semigroups:
\[
A_{-1} S(t) u = (T_{-1}(t) - I) Bu
\]
for $u \in U$ and $t \geq 0$.
Notice that $u $ does not depend on the time $t$.
Take $\lambda \in \varrho(A) = \varrho(A_{-1})$ arbitrarily. Then
\begin{align}
\label{eq:St_rep}
S(t) u = \big(I - T(t) \big)(\lambda I - A_{-1})^{-1}Bu  + 
\lambda \int^t_0 T(s) (\lambda I - A_{-1})^{-1}Bu\mathrm{d}s
\end{align}
for all $u \in U$ and all $t \geq 0$.
In \eqref{eq:St_rep}, $(\lambda I - A_{-1})^{-1}B \in \mathcal{L}(U,X)$, that is,
the resolvent $(\lambda I - A_{-1})^{-1}$ smoothens and absorbs the 
unboundedness of $B$.
The above observation 
is also applicable to the case where $X$ and $U$ are Banach spaces,
and one can obtain the next lemma.

\begin{lemma}
	\label{lem:S_bounded}
	Let $X$ and $U$ be Banach spaces and let $(T(t))_{t\geq 0}$
	be a strongly continuous semigroup on $X$.
	For all $B \in \mathcal{L}(U,X_{-1})$ and all $t \geq 0$,
	the operator $S(t)$ defined by \eqref{eq:S_def} belongs to
	$\mathcal{L}(U,X)$.
	Moreover, $\lim_{t \downarrow 0} \|S(t)u\| = 0$
	for every $u \in U$.
\end{lemma}

Let $D \in \mathcal{S}^{\rm DS}_{t_0} $
for some $t_0 >0$.
To solve the abstract evolution equation \eqref{eq:plant},
we define a function $x$ recursively by
\begin{equation}
\label{eq:unique_solution}
\left\{
\begin{aligned}
		x(0) &= x^0, \\
x(k\tau+t) &= T_D(t) x(k\tau) + S_{\!D}(t) Fx(k\tau),\quad 
t \in (0, \tau],~k \in \mathbb{Z}_+,
\end{aligned}
\right.
\end{equation}
where
\begin{equation}
\label{eq:SD_def}
S_{\!D}(t)u: = 
\int^t_0 (T_D)_{-1}(s)Bu \mathrm{d}s,\quad u \in U,~t\geq 0.
\end{equation}
By Lemma~\ref{lem:S_bounded}, $S_{\!D}(\tau) \in \mathcal{L}(U,X)$.
This implies $x (k\tau) \in X$, and therefore
$Fx(k\tau)$ is well defined for all $k \in \mathbb{Z}_+$.
Since
\[
S_{\!D}(t_2) -  S_{\!D}(t_1) = 
T_D(t_1)  S_{\!D}(t_2-t_1)
\]
for every $t_2> t_1 \geq 0$,
it follows from Lemma~\ref{lem:S_bounded} that 
for all $u \in U$,
the function
\begin{align*}
\xi_{D,u} 
&: \mathbb{R}_+ \to X\\
&: t \mapsto S_{\!D}(t)u
\end{align*}
is continuous on $\mathbb{R}_+$
with respect to $X$. 
Hence the function $x$ defined by \eqref{eq:unique_solution}
satisfies $x \in C(\mathbb{R}_+,X)$. Moreover, applying standard 
results in the semigroup theory
to the extended semigroup $((T_D)_{-1}(t))_{t\geq 0}$, we see that 
the function $x$ given in \eqref{eq:unique_solution} satisfies
\begin{equation*}
x|_{[k\tau,(k+1)\tau]} \in C^1([k\tau,(k+1)\tau],X_{-1})
\end{equation*}
and the differential equation interpreted in $X_{-1}$
\begin{equation*}
\dot x(t) = (A_{-1}+D)x(t) +BFx(k\tau)
\end{equation*}
for all $t \in (k\tau,(k+1)\tau)$ and $k \in \mathbb{Z}_+$
with initial condition $x(0) = x^0$. Clearly,
a function with these properties  is unique.
Therefore, we say that 
the function $x$ defined by \eqref{eq:unique_solution} is 
the solution of the abstract evolution equation \eqref{eq:plant}.

We conclude this preliminary section with the definition of
 the exponential stability of
the sampled-data system \eqref{eq:plant}.
\begin{definition}[Exponential stability]
	\em{
	The sampled-data system \eqref{eq:plant} with 
	$D \in \mathcal{S}^{\rm DS}_{t_0} $
	for some $t_0 >0$ is {\em exponentially stable with  decay rate 
	greater than $\omega \geq 0$} if 
	there exist constants 
	$M \geq 1$ and $\widetilde \omega >\omega$ such that 
	the solution $x$ given by \eqref{eq:unique_solution} satisfies
	\begin{equation*}
	\|x(t)\| \leq M e^{-\widetilde \omega t} \|x^0\| \qquad 
	\forall x^0 \in X,~\forall t \geq 0. 
	\end{equation*}
	}
\end{definition}

\section{Robustness analysis of exponential stability}
\label{sec:main_result}
The following theorem is the main result of this paper.
\begin{theorem}
	\label{thm:main_result}
	Let $X$ and $U$ be Banach spaces and
	let $A$ be the generator of a strongly continuous semigroup $(T(t))_{t\geq0}$
	on $X$. Suppose that $B \in \mathcal{L}(U,X_{-1})$, 
	$F \in \mathcal{L}(X,U)$, and $D_0 \in \mathcal{S}^{\rm DS}_{t_0} $
	for some $t_0 >0$.
	If the nominal sampled-data system \eqref{eq:plant}  with $D=0$
	is exponentially stable with decay rate greater than $\omega \geq 0$, then 
	there exists $c^*>0$ such that for every $c \in [0,c^*]$,
	the perturbed sampled-data system \eqref{eq:plant}  with $D = cD_0$
	is also exponentially stable with decay rate greater than $\omega$.
\end{theorem}

In the remainder of this section, we give the proof of Theorem~\ref{thm:main_result}.
Let $D \in \mathcal{S}^{\rm DS}_{t_0} $
for some $t_0 >0$.
For $t\geq 0$, define the operators $\Delta(t), \Delta_D(t) \in \mathcal{L}(X)$ by
\[
\Delta(t) := T(t)+S(t)F,\quad \Delta_D(t) := T_D(t)+S_{\!D}(t)F.
\]
Then the solution $x$ defined by \eqref{eq:unique_solution} satisfies
\[
x\big((k+1)\tau\big) = \Delta_D(\tau) x(k\tau)\qquad \forall k \in \mathbb{Z}_+.
\]
We call $\Delta_D(\tau)$ the {\em closed-loop operator}
of the discretized system.

Recall that 
an operator $\Delta \in \mathcal{L}(X)$ is said to be {\em power stable}
if there exist constants $M \geq 1$ and $\theta \in (0,1)$ such that 
$\|\Delta^k\|_{\mathcal{L}(X)}  \leq M \theta^k$ for
all $k \in \mathbb{Z}_+$.
Lemma~\ref{lem:ex_po_stability} below connects
the exponential stability of the sampled-data system 
to the power stability of the closed-loop operator 
$\Delta_D(\tau)$. This result 
can be obtained by
a slightly modification of the proof
of Proposition 2.1 in \cite{Rebarber1998}.
\begin{lemma}
	\label{lem:ex_po_stability}
	Let $X$ and $U$ be Banach spaces and 
	let $A$ be the generator of a strongly continuous semigroup $(T(t))_{t\geq0}$
	on $X$. Suppose that $B \in \mathcal{L}(U,X_{-1})$, 
	$F \in \mathcal{L}(X,U)$, and $D \in \mathcal{S}^{\rm DS}_{t_0} $
	for some $t_0>0$.
	For any $\tau>0$, the 
	sampled-data system \eqref{eq:plant} is exponentially stable with decay rate greater than $\omega \geq 0$ 
	if and only if $e^{\omega \tau}\Delta_D(\tau)$ is power stable.
\end{lemma}

From the result on the stability radius of a discrete-time system developed 
in Corollary~4.5 of \cite{Wirth1994},
we obtain the following simple result.
\begin{lemma}
	\label{lem:power_stable_perturb}
	Let $X$ be a Banach space and let $\kappa >0$. Suppose that 
	$\Delta_1 \in \mathcal{L}(X)$. If 
	$\kappa \Delta_1 $ is power stable, then
	there exists $\varepsilon >0$ such that $\kappa \Delta_2$ is also power stable for
	every $\Delta_2 \in \mathcal{L}(X)$ satisfying 
	$\|\Delta_2 - \Delta_1\|_{\mathcal{L}(X)}< \varepsilon$.
\end{lemma}

By Lemmas~\ref{lem:ex_po_stability} and \ref{lem:power_stable_perturb},
it is enough to show that $\|\Delta_D(\tau) - \Delta(\tau)\|_{\mathcal{L}(X)}$ is sufficiently small.
In what follows, we investigate the difference 
of the closed-loop operators, $\Delta_D(\tau) - \Delta(\tau)$, by 
splitting it into two parts:
\[
\Delta_D(\tau) - \Delta(\tau) = \big(T_D(\tau) - T(\tau) \big) + 
\big(S_{\!D}(\tau) - S(\tau)\big) F.
\]

First we study the difference of the semigroups, $T_D(\tau) - T(\tau)$.
\begin{lemma}
	\label{lem:TD_conv}
	Let 
	$A$ be the generator of a strongly continuous semigroup $(T(t))_{t\geq0}$
	on a Banach space $X$.
	Let $0 \leq c \leq 1$ and 
	$D_0 \in \mathcal{S}^{\rm DS}_{t_0} $
	for some $t_0 >0$, and define
	$D:=cD_0$. Then 
	the strongly continuous semigroup $(T_D(t))_{t\geq0}$
	generated by the operator $A_D$ given in Theorem~\ref{thm:generation}
	satisfies the following two properties
	for all $\tau >0$:
	\begin{enumerate}
	\item 
	$\sup \hspace{2pt}
	\{ \|T_D(t)\|_{\mathcal{L}(X)}: 0 \leq c \leq 1,~0\leq t \leq \tau   \} < \infty$;
	\item
	$\displaystyle
	\lim_{c \downarrow 0} \sup_{0\leq t \leq \tau} \|T_D(t) - T(t)\|_{\mathcal{L}(X)} = 0$.
	\end{enumerate}
\end{lemma}
\begin{proof}
	Take $\tau >0$.
	By the strong continuity of $(T(t))_{t\geq 0}$,
	there exists $M \geq 1$ such that $\|T(t)\| \leq M$ for all 
	$t \in [0,\tau]$.
	Let $q := \|V_{D_0}\|_{\mathcal{L}(\mathcal{X}_{t_0})}  < 1$. Then $\|V_D\|_{\mathcal{L}(\mathcal{X}_{t_0})}  = cq$ for 
	$D = cD_0$ with $0\leq c \leq 1$.
	Since $T_D(t) = [(I-V_D)^{-1}T](t)$ for $t \in [0,t_0]$ by the 
	variation of constants formula \eqref{eq:VCF}, 
	we obtain
	\[
	T_D - T = \big((I - V_D)^{-1}-I
	\big)T = (I - V_D)^{-1}V_DT\text{~~on~~$[0,t_0]$}.
	\]
	Hence
	\[
	\|T_D(t) - T(t)\|_{\mathcal{L}(X)} 
	\leq \frac{Mcq }{1-cq} =:q_0\qquad \forall t \in [0,t_0].
	\]
	
	Let $n \in \mathbb{N}$ satisfy $nt_0 \leq \tau < (n+1)t_0$.
	Suppose that
	for $m \in \mathbb{N}$ with $m \leq n$,
	$q_{m-1} >0$ satisfies
	\[
	\big\|T_D\big((m-1)t_0+t\big) - T\big((m-1)t_0+t\big)\big\|_{\mathcal{L}(X)}  \leq
	q_{m-1} \qquad \forall t \in [0,t_0].
	\]
	Since 
	\[
	\|T_D(t)\|_{\mathcal{L}(X)}  \leq q_0 +M \qquad \forall t \in [0,t_0],
	\]
	we obtain
	\begin{align*}
	&\|T_D(mt_0+t) - T(mt_0+t) \|_{\mathcal{L}(X)}  \\
	&\quad \leq
	\|
	T_D(t) T_D(mt_0) - T_D(t)T(mt_0)
	\|_{\mathcal{L}(X)}  +
	\|
	T_D(t) T(mt_0) - T(mt_0+t)
	\|_{\mathcal{L}(X)}   \\
	&\quad \leq 
	\|T_D(t)\|_{\mathcal{L}(X)}  ~\! \|T_D(mt_0) - T(mt_0) \|_{\mathcal{L}(X)}  + 
	\|T_D(t) - T(t)\|_{\mathcal{L}(X)}  ~\!\|T(mt_0)\|_{\mathcal{L}(X)}  \\
	&\quad \leq
	(q_0+M)q_{m-1}+Mq_0\qquad \forall t \in [0,t_0].
	\end{align*}
	Therefore, if we set 
	$q_m := (q_0+M)q_{m-1}+Mq_0$ for $m \in \mathbb{N}$,
	then $q_{m-1} \leq q_m$ and
	\[
	\|T_D(t) - T(t) \|_{\mathcal{L}(X)}  \leq q_n \qquad \forall t \in [0,\tau ].
	\]
	Since
	\[
	\max_{0\leq c \leq 1}q_0 \leq \frac{Mq}{1-q},\quad 
	\lim_{c \downarrow 0}q_0 = 0,
	\]
	it follows that 
	$\sup \hspace{2pt} \{ \|T_D(t)\|_{\mathcal{L}(X)} : 0 \leq c \leq 1,~0\leq t \leq \tau   \} < \infty$
	and
	\[
	\sup_{0\leq t \leq \tau}\|T_D(t) - T(t)\|_{\mathcal{L}(X)}  \to 0
	\]
	as $c \downarrow 0$.
\end{proof}

Using Lemma~\ref{lem:TD_conv}, we next estimate $\|S_{\!D}(\tau) - 
S(\tau)\|_{\mathcal{L}(U,X)}$.
\begin{lemma}
	\label{lem:SD_conv}
	Let $X$ and $U$ be Banach spaces and 
	let $A$ be the generator of a strongly continuous semigroup $(T(t))_{t\geq0}$
	on $X$. Suppose that  $B \in \mathcal{L}(U,X_{-1})$ and
	$F \in \mathcal{L}(X,U)$.
	Let $0 \leq c \leq 1$ and 
	$D_0 \in \mathcal{S}^{\rm DS}_{t_0} $
	for some $t_0 >0$, and define
	$D:=cD_0$. Then 
	the operator $S_{\!D}(t) \in \mathcal{L}(U,X)$ defined by \eqref{eq:SD_def}
	satisfies
	the following two properties for all $\tau  >0$:
	\begin{enumerate}
	\item
	$\sup \hspace{2pt}\{ \|S_{\!D}(t)\|_{\mathcal{L}(U,X)} : 0 \leq c \leq 1,~0\leq t \leq \tau   \} < \infty$;
	\item
	$\displaystyle
	\lim_{c \downarrow 0} \sup_{0\leq t \leq \tau} \|S_{\!D}(t) - S(t)\|_{\mathcal{L}(U,X)} = 0$.
	\end{enumerate}
\end{lemma}
\begin{proof}
	1.~
	Lemma~\ref{lem:RD_bound} shows that 
	there exists  $\lambda \in \varrho(A) =  \varrho(A_{-1})$ such that
	\begin{equation*}
	\alpha :=\|(\lambda I - A_{-1})^{-1}D_0\|_{\mathcal{L}(X)}  < 1.
	\end{equation*}
	For $D := cD_0$ with $0\leq c \leq 1$, we obtain
	\begin{equation}
	\label{eq:resol_D}
	\|(\lambda I - A_{-1})^{-1}D\|_{\mathcal{L}(X)}  = c\alpha \leq \alpha.
	\end{equation}
		From \eqref{eq:AD_identity},
	we have that for all $x \in X$,
	\begin{align*}
		(\lambda I - (A_D)_{-1})x  &= 
		(\lambda I - A_{-1} - D)x \\
		&= 
		(\lambda I - A_{-1})(I- (\lambda I - A_{-1})^{-1}D)x.
	\end{align*}
	By \eqref{eq:resol_D}, the operator $I- (\lambda I - A_{-1})^{-1}D$
	is invertible in $\mathcal{L}(X)$, and we obtain
	\[
	\|(I- (\lambda I - A_{-1})^{-1}D)^{-1}\|_{\mathcal{L}(X)}  \leq \frac{1}{1-\alpha}.
	\]
	Hence, $\lambda I - (A_D)_{-1}$ is invertible in $\mathcal{L}(X,X_{-1})$,
	\begin{equation}
	\label{eq:resol_AD_resol_A}
	(\lambda I - (A_D)_{-1})^{-1} = 
	(I- (\lambda I - A_{-1})^{-1}D)^{-1} (\lambda I - A_{-1})^{-1},
	\end{equation}
	and 
	\begin{align}
		\|(\lambda I - (A_D)_{-1})^{-1} B\|_{\mathcal{L}(U,X)}  
		&\leq 
		\frac{\|(\lambda I - A_{-1})^{-1}B\|_{\mathcal{L}(U,X)} }{1-\alpha} .
		\label{eq:resol_B_bound}
	\end{align}
	
	Since $(A_D)_{-1}$ is the generator of $((T_D)_{-1}(t))_{t\geq0}$,
	we obtain
	\begin{equation}
	\label{eq:AD_-1_cond}
	(A_D)_{-1} S_{\!D}(t) u = \big((T_D)_{-1}(t) - I
	\big) Bu\qquad \forall u \in U,~\forall t \geq 0.
	\end{equation}
	For all $u \in U $ and all $t \geq 0$,
	\eqref{eq:AD_-1_cond} yields
	\[
	(\lambda I - (A_D)_{-1}) S_{\!D}(t) u =
	\big(I - (T_D)_{-1}(t) \big) Bu + 
	\lambda \int^t_0 (T_D)_{-1}(s)Bu \mathrm{d}s
	\]
	and therefore
	\begin{align}
	S_{\!D}(t) u &= \big(I - T_D(t)\big)(\lambda I - (A_D)_{-1})^{-1}Bu + 
	\lambda \int^t_0 T_D(s) (\lambda I - (A_D)_{-1})^{-1}Bu\mathrm{d}s.
	\label{eq:SD_rep}
	\end{align}
	Combining this with \eqref{eq:resol_B_bound} and Lemma~\ref{lem:TD_conv}
	yields the first assertion.

	2.~
	Substituting $D = 0$ into \eqref{eq:SD_rep}, we have that 
	for all $u \in U$ and all $t \geq 0$,
	\begin{align*}
	S(t) u &= \big(I - T(t) \big)(\lambda I - A_{-1})^{-1}Bu  + 
	\lambda \int^t_0 T(s) (\lambda I - A_{-1})^{-1}Bu\mathrm{d}s.
	\end{align*}
	Hence
	\begin{align}
	S_{\!D}(t) u - S(t)u &=
	\big(I - T_D(t)\big)(\lambda I - (A_D)_{-1})^{-1}Bu - \big(I - T(t)\big)(\lambda I - A_{-1})^{-1}Bu 
	\notag \\
	&+ \lambda \int^t_0 
	\left(
	T_D(s) (\lambda I - (A_D)_{-1})^{-1}Bu -
	T(s) (\lambda I - A_{-1})^{-1}Bu	\right)\mathrm{d}s
	\label{eq:SD_error}
	\end{align}
	for all $u \in U$ and all $t \geq 0$. 	Moreover,
	\begin{align}
		&\| T_D(t)(\lambda I - (A_D)_{-1})^{-1}B - T(t)(\lambda I - A_{-1} )^{-1}B \|_{\mathcal{L}(U,X)}   \notag \\
		&\quad \leq
		\| T_D(t) - T(t) \|_{\mathcal{L}(X)}  ~\! \| (\lambda I - (A_D)_{-1})^{-1}B \|_{\mathcal{L}(U,X)}  
		\notag  \\
		&\qquad +
		\|T(t)\|_{\mathcal{L}(X)}  ~\!  
		\| (\lambda I - (A_D)_{-1})^{-1}B - (\lambda I - A_{-1})^{-1}B \|_{\mathcal{L}(U,X)} \label{eq:TD_triangle_ineq}
	\end{align}
	for all $t \geq 0 \notag$.
	By \eqref{eq:resol_AD_resol_A},
	\begin{align*}
	&(\lambda I - (A_D)_{-1})^{-1} - 	(\lambda I - A_{-1} )^{-1} \\
	 &\qquad =
	(\lambda I - A_{-1})^{-1}D ( I - (\lambda I - A_{-1})^{-1}D)^{-1} 	(\lambda I - A_{-1} )^{-1}.
	\end{align*}
	Since \eqref{eq:resol_D} yields
	\begin{equation*}
	\|(\lambda I - A_{-1})^{-1}D ( I - (\lambda I - A_{-1})D)^{-1} \|_{\mathcal{L}(X)} \leq
	\frac{c\alpha }{1-c\alpha },
	\end{equation*}
	it follows that 
	\begin{equation}
	\label{eq:S_bound1}
	\| (\lambda I - (A_D)_{-1})^{-1}B - (\lambda I - A_{-1} )^{-1}B \|_{\mathcal{L}(U,X)}
	\leq \frac{c\alpha }{1-c\alpha } \|(\lambda I - A_{-1} )^{-1}B\|_{\mathcal{L}(U,X)}.
	\end{equation}

	Take $\varepsilon >0$ and $\tau  > 0$.
	Combining Lemma~\ref{lem:TD_conv} with
	the estimates  \eqref{eq:resol_B_bound}, \eqref{eq:TD_triangle_ineq}, and \eqref{eq:S_bound1},
	we see 
	that there exists $c^* > 0$
	such that
	for all $t \in [0,\tau ]$ and all $D=cD_0$ with $0\leq c \leq c^*$,
	\begin{align}
	\label{eq:S_bound2}
	\| T_D(t)(\lambda I - (A_D)_{-1})^{-1}B - T(t)(\lambda I - A_{-1} )^{-1}B \|_{\mathcal{L}(U,X)} 
	<
	\varepsilon.
	\end{align}
	This implies that if $u \in U$ satisfies $\|u\|_U \leq 1$, then
	for all $t \in [0,\tau ]$ and all
	$D=cD_0$ with $0\leq c \leq c^*$, 
	\begin{align}
	\label{eq:S_bound3}
	&\left\|\int^t_0 \left(T_D(s) (\lambda I - (A_D)_{-1})^{-1}Bu -
	T(s) (\lambda I - A_{-1})^{-1}Bu \right) \mathrm{d}s \right\|< \tau  \varepsilon.
	\end{align}
	Applying  \eqref{eq:S_bound2} and \eqref{eq:S_bound3} to 
	\eqref{eq:SD_error},
	we have that for all $t \in [0,\tau ]$ and all
	$D=cD_0$ with $0\leq c \leq c^*$,
	\[
	\|S_{\!D}(t) - S(t)\|_{\mathcal{L}(U,X)} < (2+|\lambda|\tau)\varepsilon.
	\]
	Since 
	$\varepsilon$ was arbitrary, 
	the second assertion follows.
\end{proof}

We are now ready to prove the main theorem.
\begin{proof}[Proof of Theorem~\ref{thm:main_result}]
	Since the nominal sampled-data system \eqref{eq:plant}  with $D=0$
	is exponentially stable with decay rate greater than $\omega \geq 0$ by assumption,
	it follows from Lemma~\ref{lem:ex_po_stability} that 
	$e^{\omega \tau} \Delta(\tau)$ is power stable.
	By Lemma~\ref{lem:power_stable_perturb},
	there exists $\varepsilon >0$ such that 
	$\|\Delta_{D}(\tau) - \Delta(\tau) \|_{\mathcal{L}(X)} < \varepsilon$
	implies the power stability of $e^{\omega \tau} \Delta_D(\tau)$.
	From Lemmas~\ref{lem:TD_conv} and \ref{lem:SD_conv}, we find 
	$c^* >0$ satisfying 
	$\|\Delta_D(\tau) - \Delta(\tau)\|_{\mathcal{L}(X)} < \varepsilon$ for all
	perturbations in the form
	$D = cD_0$ with $c \in [0,c^*]$.
	Finally, using Lemma~\ref{lem:ex_po_stability} again,
	we have that for all $c \in [0,c^*]$, 
	the perturbed sampled-data system \eqref{eq:plant}  with 
	$D=cD_0$
	is exponentially stable with decay rate greater than $\omega$.
\end{proof}

 \section{Examples}
 \label{sec:example}
 \subsection{Heat equation with boundary perturbation}
 \label{sec:heat_equation}
Consider a metal rod of length one. Let $z(\xi,t)$
be the temperature of the rod at position $\xi \in [0,1]$ and
at time $t \geq 0$.
We control the temperature of the rod by
means of a heat flux acting at one boundary $\xi =1$, and
the input is generated by a sampled-data  state-feedback controller with
sampling period $\tau>0$. Moreover,
a certain perturbation effect appears at the other boundary $\xi=0$.
The system is described by the following partial differential equation:
\begin{equation}
\label{eq:Heat_equation}
\left\{
\begin{aligned}
\frac{\partial z}{\partial t} (\xi,t) &= 
\frac{\partial^2 z}{\partial \xi^2} (\xi,t), \quad& &0\leq \xi \leq 1,~t \geq 0\\
 z(\xi,0) &= z^0(\xi),\quad &&0\leq \xi \leq 1  \\
\frac{\partial z}{\partial \xi} (1,t)  &= Fz(\cdot, t), \quad &&k\tau \leq t < (k+1)\tau,~k \in \mathbb{Z}_+\\
\frac{\partial z}{\partial \xi} (0,t)  &= Hz(\cdot, t),\quad&& t \geq 0,
\end{aligned}
\right.
\end{equation}
where $F$ and $H$ are functionals on the space of all functions 
mapping $[0,1]$ into $\mathbb{R}$.

First, we transform the heat equation \eqref{eq:Heat_equation} into
the abstract evolution equation \eqref{eq:plant}.
Let the state space $X$ and the input space $U$ be
$X = L^2(0,1)$ and $U = \mathbb{C}$. 
We assume that $F,H \in \mathcal{L}(X,\mathbb{C})$.
The generator $A$ is given by
\begin{equation*}
	Af = f''
\end{equation*}
with domain 
\begin{align*}
\dom(A) = \{f\in L^{2}(0,1): \text{$f$, $f'$ are absolutely continuous,}\\
\text{$f'' \in L^{2}(0,1)$, and~} f'(0) = f'(1) = 0 \}.
\end{align*}
Let $(T(t))_{t\geq 0}$ be the strongly continuous semigroup on $X$
generated by $A$.
The control operator $B$ and the perturbation $D$ are given by
\begin{align*}
\begin{aligned}
Bu &= \delta_1 u,\quad u\in \mathbb{C} \\
Df &= -\delta_0 Hf,\quad f \in X,
\end{aligned}
\end{align*}
where $\delta_\xi$ is the delta function supported at the point $\xi \in [0,1]$;
see \cite{Ho1983} and \cite[Chap.~10]{Tucsnak2009}
for the derivation of these operators. Since $\delta_1$ and $\delta_0$
belong to the extrapolation space 
$X_{-1}$, it follows that $B\in \mathcal{L}(\mathbb{C},X_{-1})$ and 
$D \in \mathcal{L}(X,X_{-1})$.

By the technique explained in Remark~10.1.4 of \cite{Tucsnak2009},
the perturbation $D$ is written as
\[
Df = (Hf)(I-A_{-1})\eta\qquad \forall f \in X,
\]
where 
\[
\eta(\xi) := -\frac{(e^{\xi} + e^2e^{-\xi})}{e^2-1},\quad
0 \leq \xi \leq 1.
\]
Using this representation of the perturbation $D$,
we see that the perturbed operator $A_D$ given in 
Theorem~\ref{thm:generation} satisfies
\begin{align*}
\dom(A_D) = 
 \{f\in L^{2}(0,1): \text{$f$, $f'$ are absolutely continuous,}\\
\text{$f'' \in L^{2}(0,1)$,}~~ f'(0) = Hf, \text{~and~~} f'(1) = 0 \}
\end{align*}
and $A_D f = f''$ for all $f \in \dom(A_D)$.

We shall show that $D$ is a Desch-Shappacher perturbation.
Define $f_n \in X$ by
\begin{align*}
	f_0(\xi) := 1,\quad 
	f_n(\xi) := \sqrt{2} \cos(n\pi \xi)
\end{align*}
for $\xi \in [0,1]$ and $n \in \mathbb{N}$.
Then $(f_n)_{n \in \mathbb{Z}_+}$ is an orthonormal basis of $X$.
The generator $A$ may be written as
\begin{equation}
\label{eq:A_expansion_heat_eq}
Af = -\sum_{n=0}^\infty n^2\pi^2 \langle f, f_n \rangle_{L^2} f_n
\qquad \forall f \in \dom(A)
\end{equation}
and 
\[
\dom(A) = \left\{
f \in L^2(0,1) : 
\sum_{n=0}^\infty n^4\pi^4 ~ \big|\langle f, f_n \rangle_{L^2}\big|^2 < \infty
\right\},
\]
where $\langle \cdot, \cdot \rangle_{L^2}$
is an inner product in $L^2(0,1)$ defined by
\[
\langle f,g \rangle_{L^2} := 
\int_0^1  f(\xi) \overline{g(\xi)} \mathrm{d}\xi,\quad f,g \in L^2(0,1).
\]
We refer the reader to, e.g., Example~3.2.15 of \cite{Curtain2020}
for this representation of $A$.
A standard application of the Carleson measure criterion developed in \cite{Ho1983,Weiss1988} yields that $\delta_0$ 
is a finite-time 
$L^2$-admissible input element for the semigroup $(T(t))_{t\geq 0}$,
i.e., there exists $t_1 >0$ such that 
\[
\int^{t_1}_0 T_{-1}(t_1-s) \delta_0 w(s)\mathrm{d}s \in X\qquad \forall w \in L^2(0,t_1).
\]
Hence
\[
\int^{t_1}_0 T_{-1}(t_1-s) D \phi(s)\mathrm{d}s \in X\qquad \forall \phi \in L^2([0,t_1],X).
\]
By Theorem~\ref{thm:verify}, $D\in \mathcal{S}^{\rm DS}_{t_0} $ for some $t_0>0$.

Let the operator $H \in \mathcal{L}(X,\mathbb{C})$ be represented as $H = cH_0$
with $0\leq c \leq 1$ and $H_0 \in \mathcal{L}(X,\mathbb{C})$.
Theorem~\ref{thm:main_result} shows that
if the feedback operator $F \in \mathcal{L}(X,\mathbb{C})$ and the sampling 
period $\tau > 0$ are chosen
so that the nominal sampled-data system with $D=0$
is exponentially stable with decay rate greater 
than $\omega\geq 0$, then
the perturbed sampled-data system also has the same stability property for
all sufficiently small $c>0$.

 \subsection{Perturbed diagonal system}
Let  $
\ell^q := \{(x_n)_{n\in \mathbb{Z}_+}\subset \mathbb{C} :
\sum_{n=0}^{\infty}|x_n|^q< \infty  \} $ for $q \in [1,\infty)$ and 
$
\ell^{\infty} := \{(x_n)_{n\in \mathbb{Z}_+}\subset \mathbb{C} :
\sup_{n \in \mathbb{Z}_+}|x_n|< \infty  \}$. 
From the expansion \eqref{eq:A_expansion_heat_eq} of the generator $A$,
the heat equation \eqref{eq:Heat_equation} 
can be regarded
as a diagonal system.
More precisely, the heat equation \eqref{eq:Heat_equation} 
may be written as
the abstract evolution equation \eqref{eq:plant} with state space
$X = \ell^2 $ 
and input space $U = \mathbb{C}$.
The generator 
$A$ is a diagonal operator on $\ell^2$ given by
\[
(Ax)_{n \in \mathbb{Z}_+} =
(-n^2\pi^2 x_n)_{n \in \mathbb{Z}_+}
\]
with domain
\[
\dom(A) = \big\{
x = (x_n)_{n\in \mathbb{Z}_+}\in \ell^2: (n^2\pi^2 x_n)_{n \in \mathbb{Z}_+} \in \ell^2
\big\}.
\]
The control operator $B$ and the perturbation $D$ are written as
\begin{align*}
B u &= bu,\quad u \in \mathbb{C}\\
Df &= 
d Hf,\quad f \in \ell^2,
\end{align*}
where $b = (b_n)_{n\in \mathbb{Z}_+} \in \ell^{\infty} $ and 
$d =(d_n)_{n\in \mathbb{Z}_+}\in \ell^{\infty} $
 are given by
 \[
 b_0 = 1,~d_0 = -1,~b_n = (-1)^n\sqrt{2},~d_n = -\sqrt{2}\qquad 
\forall n \in \mathbb{N}
 \]
 and $H \in \mathcal{L}(\ell^2,\mathbb{C})$.
Here we consider a more general case $X = \ell^q$
with $q \in (1,\infty)$.
 
 Let $\kappa,\gamma >0$ and $\zeta \in \mathbb{R}$. Set
 $\lambda_n := -\kappa n^\gamma + \zeta$ for $n \in \mathbb{Z}_+$.
 Define a diagonal operator $A$ on $\ell^q$ by
 \[
 (Ax)_{n \in \mathbb{Z}_+} :=
 (\lambda_n  x_n)_{n \in \mathbb{Z}_+}
 \]
 with domain
 \[
\dom (A) := \big\{
 x = (x_n)_{n\in \mathbb{Z}_+}\in \ell^q: (\lambda_n x_n)_{n \in \mathbb{Z}_+} \in \ell^q
 \big\}.
 \]
Let $b,d \in \ell^{\infty} $ and $H \in \mathcal{L}(\ell^q,\mathbb{C})$.
Define
\begin{align*}
	B u &:= bu,\quad u \in \mathbb{C}\\
	Df &:= d Hf,\quad f \in \ell^q.
\end{align*}

Applying to the diagonal system
the Carleson measure criterion developed in 
Theorem~3.2 of \cite{Haak2010}, one can obtain the following fact:
Suppose that 
$q \in (1,\infty)$ and $\gamma >0$ satisfy
\begin{equation}
\label{eq:q_cond}
q \geq \frac{\gamma + 1}{\gamma}.
\end{equation}
Define 
\[
p :=\frac{\gamma q}{\gamma q - 1} \leq q.
\]
Then
$b$ and $d$ belong to the extrapolation space $X_{-1}$.
Moreover,
$d$ is a finite-time $L^{p}$-admissible 
input element for the semigroup
$(T(t))_{t \geq 0}$ generated by $A$,
i.e., there exists $t_1 >0$ such that 
\begin{equation}
\label{eq:d_admissible}
\int^{t_1}_0 T_{-1}(t_1-s) d w(s)\mathrm{d}s \in X\qquad \forall w \in 
L^{p}(0,t_1).
\end{equation}

By Theorem~\ref{thm:verify} and \eqref{eq:d_admissible}, we obtain
$D\in \mathcal{S}^{\rm DS}_{t_0} $ for some $t_0>0$.
Thus, Theorem~\ref{thm:main_result} 
can be applied to the perturbed 
diagonal system as in Section~\ref{sec:heat_equation}.
 
\section{Conclusion}
We have studied the exponential stability of
infinite-dimensional sampled-data systems with
unbounded control operators and Desch-Schappacher perturbations.
Since
the exponential stability of the sampled-data system
is equivalent to the power stability of the closed-loop
operator of the discretized system, we have investigated the discretized system
for the stability analysis of the sampled-data system.
To treat the unbounded control operator and perturbation,
we have utilized the property that 
their products with the resolvent of the generator map
boundedly into the state space.
Future work involves the  analysis
of the strong stability of
perturbed infinite-dimensional sampled-data systems.

\section*{Acknowledgments}
The author 
would like to thank Professor Hideki Sano for helpful comments on
perturbed heat equations.
The author is also grateful to the associate editor who made 
some positive suggestions for improving this manuscript, in particular,
the inclusion of the example of a diagonal system in a Banach framework.










\begin{thebibliography}{10}
	\expandafter\ifx\csname url\endcsname\relax
	\def\url#1{\texttt{#1}}\fi
	\expandafter\ifx\csname urlprefix\endcsname\relax\def\urlprefix{URL }\fi
	\expandafter\ifx\csname href\endcsname\relax
	\def\href#1#2{#2} \def\path#1{#1}\fi
	
	\bibitem{Hille1957}
	E.~Hille, R.~S. Phillips, Functional Analysis and Semigroups, Amer. Math. Soc.
	Coll. Publ., vol. 31, Amer. Math. Soc., 1957.
	
	\bibitem{Engel2000}
	K.-J. Engel, R.~Nagel, One-Parameter Semigroups for Linear Evolution Equations,
	New York: Springer, 2000.
	
	\bibitem{Miyadera1966}
	I.~Miyadera, On perturbation theory for semi-groups of operators, T\^ohoku
	Math. J. 18 (1966) 299--310.
	
	\bibitem{Voigt1977}
	J.~Voigt, On the perturbation theory for strongly continuous semigroup, Math.
	Ann. 229 (1977) 163--171.
	
	\bibitem{Desch1989}
	G.~Desch, W.~Schappacher, Some generation results for perturbed semigroups, in:
	P.~Cl\'ement, S.~Invernizzi, E.~Mitidieri, I.~I. Vrabie (Eds.), Semigroup
	Theory and Applications (Proceedings, Trieste, 1987), Vol. 116 of Lect. Notes
	in Pure and Appl. Math., Marcel Dekker, 1989, pp. 125--152.
	
	\bibitem{Weiss1994}
	G.~Weiss, Regular linear systems with feedback, Math. Control Signals Systems 7
	(1994) 23--57.
	
	\bibitem{Staffans2005}
	O.~J. Staffans, Well-Posed Linear Systems, Cambridge, UK: Cambridge Univ.
	Press, 2005.
	
	\bibitem{Haak2006}
	B.~H. Haak, M.~Haase, P.~C. Kunstmann, Perturbation, interpolation, and maximal
	regularity, Adv. Differential Equations 11 (2006) 201--240.
	
	\bibitem{Pritchard1989}
	A.~J. Pritchard, S.~Townley, Robustness of linear systems, J. Differential Equations
	77 (1989) 254--286.
	
	\bibitem{Pandolfi1991}
	L.~Pandolfi, H.~Zwart, Stability of perturbed linear distributed parameter
	systems, Systems Control Lett. 17 (1991) 257--264.
	
	\bibitem{Townley1992Survey}
	S.~Townley, {Robust stability radii for distributed parameter systems: A
		survey}, in: R.~F. Curtain, A.~Bensoussan, J.~L. Lions (Eds.), Analysis and
	Optimization of Systems: State and Frequency Domain Approaches for
	Infinite-Dimensional Systems (Proceedings, Sophia-Antipolis, 1992), Vol. 185
	of Lect. Notes in Control and Inf. Sci., Berlin, Heidelberg: Springer, 1993,
	pp. 302--313.
	
	\bibitem{PaunonenMasterThesis}
	L.~Paunonen,
	\href{https://www.researchgate.net/publication/255180229_Robustness_of_Stability_of_C_0-Semigroups}{{Robustness
			of Stability of $C_0$-Semigroups}}, Master's thesis, Tampere University of
	Technology, 2007.

	
	\bibitem{Wirth1994}
	F.~Wirth, D.~Hinrichsen, On stability radii of infinite-dimensional
	time-varying discrete-time systems, IMA J. Math. Control Inform. 11 (1994)
	253--276.
	
	\bibitem{Sasu2004}
	B.~Sasu, A.~L. Sasu, Stability and stabilizability for linear systems of
	difference equations, J. Difference Equations Appl. 10 (2004) 1085--1105.
	
	\bibitem{Paunonen2011}
	L.~Paunonen, {Perturbation of strongly and polynomially stable Riesz-spectral
		operators}, Systems Control Lett. 60 (2011) 234--248.
	
	\bibitem{Paunonen2012SS}
	L.~Paunonen, Robustness of strongly and polynomially stable semigroups, J.
	Funct. Anal. 263 (2012) 2555--2583.
	
	\bibitem{Paunonen2014JDE}
	L.~Paunonen, Robustness of strong stability of semigroups, J. Differential Equations
	257 (2014) 4403--4436.
	
	\bibitem{Paunonen2015}
	L.~Paunonen, Robustness of strong stability of discrete semigroups, Systems
	Control Lett. 75 (2015) 35--40.
	
	\bibitem{Rebarber2002}
	R.~Rebarber, S.~Townley, {Nonrobustness of closed-loop stability for
		infinite-dimensional systems under sample and hold}, IEEE Trans. Automat.
	Control 47 (2002) 1381--1385.
	
	\bibitem{Logemann2003}
	H.~Logemann, R.~Rebarber, S.~Townley, Stability of infinite-dimensional
	sampled-data systems, Trans. Amer. Math. Soc. 355 (2003) 3301--3328.
	
	\bibitem{Rebarber2006}
	R.~Rebarber, S.~Townley, {Robustness with respect to sampling for stabilization
		of Riesz spectral systems}, IEEE Trans. Automat. Control 51 (2006)
	1519--1522.
	
	\bibitem{Wakaiki2021SIAM}
	M.~Wakaiki, {Strong stability of sampled-data Riesz-spectral systems}, SIAM J.
	Control Optim. 59 (2021) 3498--3523.
	
	\bibitem{Wakaiki2020SCL}
	M.~Wakaiki, Y.~Yamamoto, Stability analysis of perturbed infinite-dimensional
	sampled-data systems, Systems Control Lett. 138 (2020) Art. no. 104652.
	
	\bibitem{Wakaiki2020_self}
	M.~Wakaiki, H.~Sano, {Stability analysis of infinite-dimensional
		event-triggered and self-triggered control systems with Lipschitz
		perturbations}, Math. Control Relat. Fields 12 (2022) 245--273.
	
	\bibitem{Piazzera2004}
	S.~Piazzera, An age-dependent population equation with delayed birth process,
	Math. Methods Appl. Sci. 27 (2004) 427--439.
	
	\bibitem{Maniar2005}
	L.~Maniar, {Stability of asymptotic properties of Hille-Yosida operators under
		perturbations and retarded differential equations}, Quaest. Math. 28 (2005)
	39--53.
	
	\bibitem{Boulite2005}
	S.~Boulite, A.~Idrissi, L.~Maniar, Robustness of controllability under some
	unbounded perturbations, J. Math. Anal. Appl. 304 (2005) 409--421.
	
	\bibitem{Hadd2005JEEQ}
	S.~Hadd, Exact controllability of infinite dimensional systems persists under
	small perturbations, J. Evol. Equations 5 (2005) 545--555.
	
	\bibitem{Tucsnak2009}
	M.~Tucsnak, G.~Weiss, Observation and Control of Operator Semigroups, Basel:
	Birkh\"auser, 2009.
	
	\bibitem{Rebarber1998}
	R.~Rebarber, S.~Townley, Generalized sampled data feedback control of
	distributed parameter systems, Systems Control Lett. 34 (1998) 229--240.
	
	\bibitem{Ho1983}
	L.~F. Ho, D.~L. Russell, {Admissible input elements for systems in Hilbert
		space and a Carleson measure criterion}, SIAM J. Control Optim. 21 (1983)
	614--640.
	
	\bibitem{Curtain2020}
	R.~F. Curtain, H.~J. Zwart, An Introduction to Infinite-Dimensional Systems, A
	State Space Approach, New York: Springer, 2020.
	
	\bibitem{Weiss1988}
	G.~Weiss, Admissibility of input elements for diagonal semigroups on $l^2$,
	Systems Control Lett. 10 (1988) 79--82.
	
	\bibitem{Haak2010}
	B.~H. Haak, {On the Carleson measure criterion in linear systems theory},
	Compl. Anal. Oper. Theory 4 (2010) 281--299.
	
\end{thebibliography}
\end{document}